\documentclass[3p,times]{elsarticle}

\usepackage{ecrc}

\volume{00}

\firstpage{1}

\journalname{}

\runauth{Chen Wang et al.}

\jid{procs}

\jnltitlelogo{}

\CopyrightLine{2024}{Published by Elsevier Ltd.}

\usepackage{amssymb}
\usepackage[figuresright]{rotating}

\usepackage{arydshln} % 提供 \hdashline
\usepackage{amsmath}
\usepackage{amsthm}

\usepackage{graphics}
\usepackage{subfigure}
\usepackage{lineno,hyperref}
\usepackage{cases}
\usepackage{bm}
\usepackage[linesnumbered,ruled,vlined]{algorithm2e}
\usepackage{hyperref}
\usepackage[noend]{algpseudocode}
\modulolinenumbers[1]
\newtheorem{definition}{Definition}
\newtheorem{theorem}{Theorem}

\newtheorem{corollary}{Corollary}

\newtheoremstyle{prestyle}
{0} % Space above
{\topsep} % Space below
{\itshape} % Body font
{} % Indent amount
{\bfseries} % Theorem head font
{.} % Punctuation after theorem head
{.5em} % Space after theorem head
{} % Theorem head spec (can be left empty, meaning `normal')
\theoremstyle{prestyle}

\theoremstyle{definition}	

\begin{document}

\begin{frontmatter}



\dochead{}
%% e.g. \dochead{17th International Conference on Dynamical Processes in Excited States of Solids}

\title{Efficient Computation for Invertibility Sequence of Banded Toeplitz Matrices}


\author[]{Chen Wang}
\ead{2120220677@mail.nankai.edu.cn}
%\author[]{Hailong Yu}
%\ead{hlyu@mail.nankai.edu.cn}
\author[]{Chao Wang\corref{cor1}}
\ead{wangchao@nankai.edu.cn}
\address[mymainaddress]{Address: College of Software, Nankai University, Tianjin 300350, China}

\cortext[cor1]{Corresponding author.}

\begin{abstract}
	
	When solving systems of banded Toeplitz equations or calculating their inverses, it is necessary to determine the invertibility of the matrices beforehand. In this paper, we equate the invertibility of an $n$-order banded Toeplitz matrix with bandwidth $2k+1$ to that of a small $k*k$ matrix. By utilizing a specially designed algorithm, we compute the invertibility sequence of a class of banded Toeplitz matrices with a time complexity of $5k^2n/2+kn$ and a space complexity of $3k^2$ where $n$ is the size of the largest matrix. This enables efficient preprocessing when solving equation systems and inverses of banded Toeplitz matrices.
\end{abstract}

\begin{keyword}
	
	
	banded Toeplitz matrix
	\sep invertibility sequence
	\sep computational complexity

\end{keyword}

\end{frontmatter}

%\linenumbers

\section{Introduction and preliminaries \label{s1}}
An $n$-order banded Toeplitz matrix with bandwidth $2k+1$ takes the form:
\begin{equation}
\label{M}
M_n = 
\begin{bmatrix}
	
	x_{0} & \cdots & x_{k} & 0 & \cdots & 0\\
	\vdots & \ddots &  & \ddots & \ddots & \vdots\\
	x_{-k} &  & \ddots &  & \ddots & 0\\
	0 & \ddots &  & \ddots & & x_{k}\\
	\vdots & \ddots & \ddots &  & \ddots & \vdots\\
	0 & \cdots & 0 & x_{-k} & \cdots & x_{0}
	
\end{bmatrix}_{n*n}.
\end{equation}
Banded Toeplitz matrices, as transition matrices for convolutions, frequently appear in the numerical solutions of partial differential equations using finite difference methods, finite element methods, and spectral methods \cite{ACETO20122960, ACETO20123857, poletti2021superfast}. They can be applied in the mathematical representation of high-dimensional nonlinear electromagnetic interference signals \cite{patil2008eficient,tang2006spectral}.

Currently, the majority of computations on banded Toeplitz matrices, including solving equation systems \cite{kumar1985fast, liu2020fast} and inverses \cite{arslan2013positive, lv2008note, wang2015explicit}, require that the banded Toeplitz matrices are non-singular, meaning the determination of their invertibility is necessary. At present, most methods determine the invertibility by calculating determinants \cite{cinkir2014fast, jia2016homogeneous, jia2023cost, kilic2008computational, talibi2018numerical}. This leads to computational waste because we are not concerned with the specific value of their determinants.

Based on the reasons mentioned above, we propose a rapid determination method for the invertibility of banded Toeplitz matrices. We equate the invertibility of an $n$-order banded Toeplitz matrix with bandwidth $2k+1$ to that of a small $k*k$ matrix and introduce a new algorithm, which can solve the invertibility sequence of a class of banded Toeplitz matrices with time complexity $5k^2n/2+kn$ and space complexity $3k^2$ where $n$ is the size of the largest matrix. Additionally, while most other studies focus on tridiagonal or pentadiagonal Toeplitz matrices, our algorithm is compatible with larger values of $k$. This enables efficient preprocessing when solving equation systems and inverses of banded Toeplitz matrices. Since the invertibility of banded Toeplitz matrices over finite fields is equivalent to the reversibility of one-dimensional null-boundary linear cellular automata \cite{MARTINDELREY20118360}, we compare our determination algorithm with the latest equivalent invertibility determination algorithms in cellular automata \cite{du2022efficient}. The results show that our algorithm has significant efficiency advantages.

\begin{table}[h]
\center
\caption{Comparison of complexity with equivalent algorithm}
\label{table1}
\begin{tabular}{p{3cm}p{3cm}}
	\hline
	algorithm & time complexity\\
	\hline
	SBP \cite{du2022efficient}  & $O(k^3n)$    \\
	ours      & $5k^2n/2+kn$    \\
	\hline
\end{tabular}
\end{table}

This paper consists of three sections. The second section presents the main theorem and algorithm of this paper. The third section summarizes the work of the entire paper.

\section{Invertibility of banded Toeplitz matrices}
The invertibility of a matrix is a simpler problem than the determinant, so we aim to completely compute the invertibility sequence for a class of banded Toeplitz matrices:

\begin{definition}
The \textbf{invertibility sequence} is a binary sequence of length $n$ where the $i$-th number indicates whether the $i$-th order banded Toeplitz matrix is invertible.
\end{definition}

Next, we will begin our calculations, assuming the bandwidth is $2k+1$. Consider the following sequence:
\begin{equation}
\label{eq1}
v_{i,j}\ (-k \le i \le n+k, 1 \le j \le k)=
\begin{cases}
	0  & \text{ for } i \le k \text{ and } i \ne j \\
	1  & \text{ for } i = j\\
	-(x_{k-1}v_{i-1,j}+x_{k-2}v_{i-2,j}+ \cdots +x_{-k}v_{i-2k,j})/x_{k} & \text{ for } i>k
\end{cases}.
\end{equation}
Let $W_i$ be a $k*k$ matrix, with its elements as described in Eq. \ref{eq2}.
\begin{equation}
\label{eq2}
W_i=\begin{bmatrix}
	v_{i+1,1} & v_{i+1,2} & \cdots & v_{i+1,k}\\
	v_{i+2,1} & v_{i+2,2} & \cdots & v_{i+2,k}\\
	\vdots & \vdots & \ddots & \vdots\\
	v_{i+k,1} & v_{i+k,2} & \cdots & v_{i+k,k}
\end{bmatrix}.
\end{equation}
We have the following important theorem:

\begin{theorem}
\label{t1}
An $n$-order banded Toeplitz matrix $M_n$ with bandwidth $2k+1$ ($n>k$) is invertible if and only if $W_n$ is invertible.
\end{theorem}
\begin{proof}
Let matrices $P$ and $Q$ be defined as follows:
\begin{equation}
	P = QW_n = 
	\begin{bmatrix}
		-x_{k} & 0 & \cdots & 0\\
		-x_{k-1} & -x_{k} & \cdots & 0\\
		\vdots & \vdots & \ddots & \vdots\\
		-x_{1} & -x_{2} & \cdots & -x_{k}
	\end{bmatrix}
	\begin{bmatrix}
		v_{n+1,1} & v_{n+1,2} & \cdots & v_{n+1,k}\\
		v_{n+2,1} & v_{n+2,2} & \cdots & v_{n+2,k}\\
		\vdots & \vdots & \ddots & \vdots\\
		v_{n+k,1} & v_{n+k,2} & \cdots & v_{n+k,k}
	\end{bmatrix}.
\end{equation}
We have the following equation, where $I_k$ is the $n$-order identity matrix. 
\begin{equation}
	\label{main}
	%M_nV = 
	M_n
	\left[\begin{array}{ccc}
		&  & \\
		& I_k & \\
		&  & \\
		\hdashline
		v_{k+1,1}& \cdots & v_{k+1,k}\\
		\vdots & \ddots & \vdots\\
		v_{n,1}& \cdots & v_{n,k}\\
	\end{array}\right]
	=
	\left[\begin{array}{ccc}
		&  & \\
		& O_{(n-k)*k} & \\
		&  & \\
		\hdashline
		&  & \\
		& P_{k*k} & \\
		&  & \\
	\end{array}\right].
\end{equation}

First, we prove the necessity: if $M_n$ is an invertible matrix, then when it is multiplied by a column full-rank matrix, the result is a column full-rank matrix. The rank of $P$ is $k$, so $P$ is invertible.

Next, we prove the sufficiency: if $W_n$ is invertible, then $P$ is also invertible. Multiplying both sides of the Eq. \ref{main} by $P^{-1}$ on the right, there exist 
$k$ vectors $U_{n-k+1},U_{n-k+2},\cdots U_{n}$ that make:
\begin{equation}
	M_n[U_{n-k+1},U_{n-k+2},\cdots, U_{n}] 
	=
	\left|\begin{array}{ccc}
		0  \\
		\hline
		I_k \\
	\end{array}\right|,
\end{equation}

At this point, we can construct a unique matrix $U= [U_{1},U_{2},\cdots,U_{n}$] such that $M_nU=I_n$, where for $i>n$, $U_i=0$ and $E_i$ is the vector of order $i$ of the canonical basis of $\mathbb{K}^n$.
\begin{equation}
	U_{i}= (E_{i-k}-x_{-k+1}U_{i+1}-x_{-k+2}U_{i+2}-\cdots-x_{k}U_{i+2k})/x_{-k} .
\end{equation}
\end{proof}

We have equated the invertibility of a banded Toeplitz matrix with bandwidth $2k+1$ to the invertibility of a small $k*k$ matrix, which significantly enhances our computational efficiency.

\begin{corollary}
The invertibility sequence of tridiagonal Toeplitz matrices can be calculated with a time complexity of $3n$.
\end{corollary}
\begin{proof}
In this case, $k=1$, which means the invertibility determining matrix $W$ consists of a single element. We only need to determine whether it is zero. The cost of calculating $W_1$ to $W_n$ using Eq. \ref{eq1} is $3n$.
\end{proof}

\begin{corollary}
The invertibility sequence of pentadiagonal Toeplitz matrices can be calculated with a time complexity of $11n$.
\end{corollary}
\begin{proof}
In this case, $k=2$,  which means the invertibility determining matrix $W$ is a $2*2$ matrix. We only need to determine if the ratio of the two elements in each row of $W$ is equal to that in the next row, which costs $n$. The cost of calculating  $W_1$ to $W_n$ using Eq. \ref{eq1} is $10n$. Therefore, the total cost is $11n$.
\end{proof}

\begin{corollary}
The invertibility sequence of $(2k+1)$-diagonal Toeplitz matrices can be calculated with time complexity $5k^2n/2+kn$ and space complexity $3k^2$.
\end{corollary}
\begin{proof}
Calculating the invertibility sequence for banded Toeplitz matrices involves determining the invertibility of $W_{1}, W_{2},\cdots W_{n}$, where each $W$ is a $k*k$ matrix. If Gaussian elimination is used for each matrix to row-echelon form, the time complexity could reach $O(k^3n)$. Since $k-1$ rows are the same between $W_{i}$ and $W_{i+1}$, this means that the computations performed on $W_{i}$ can be utilized for $W_{i+1}$. This reuse of calculations can optimize the overall process and reduce the computational cost.

Consider the initial situation, set $W_{0}=Y_{0}=I$, where $I$ is the identity matrix, a row-echelon form of a full-rank matrix. By removing the first row of $W_{0}$ and adding the bottom row of $W_1$, a new matrix $Y_{1}$ is formed. It can be seen that the invertibility of $W_{1}$ and $Y_{1}$ is the same. At this point, the first $k-1$ rows of $Y_{1}$ are \textbf{quasi-row-echelon which can be transformed into row-echelon form through row swapping}. Assuming that $index[a]$ is the index of the first non-zero element from the left in the $a$-th row of each $Y$, there are two scenarios:
\begin{itemize}
	\item If $\forall i \in \mathbb{Z}, 1 \leq i \leq k$, $index[k] \ne index[i]$, this indicates that $Y_{i}$ is quasi-row-echelon. In this scenario, we can quickly determine the invertibility of $Y_1$, whether there exists a row that is entirely zero.
	
	\item If $\exists i \in \mathbb{Z}, 1 \leq i \leq k$, $index[k] = index[i]$, then use the $k$-th row to perform Gaussian elimination on the $i$-th row ($k>i$). After Gaussian elimination, $index[i]$ will grow and be updated. If at this point, $\exists j \in \mathbb{Z}, 1 \leq j \leq k$, $index[j] = index[i]$, then use the lower row to perform Gaussian elimination on the upper one. Continue this process until $Y_{1}$ achieves a quasi-row-echelon form. Then, we can quickly determine the invertibility of $Y_{1}$.
\end{itemize}
\begin{algorithm}[h]
	\footnotesize
	\SetAlgoLined
	\label{a1}
	\KwData{vector $X=[x_{-k},\cdots,x_{0},\cdots,x_{k}]$ and the size $n$ of the largest matrix}
	\KwResult{invertibility sequence $R$}
	\tcp{\textbf{Step1}: initialize $W_0$ and $Y_0$}
	\For{$i=1$;$i\leq k$;$i++$}{
		\For{$j=1$;$j\leq 2k$;$j++$}{
			\eIf{$j=i+k$}{
				$W[j][i]=1$\;
			}{
				$W[j][i]=0$\;
			}			
		}
		\For{$j=1$;$j\leq k$;$j++$}{
			\eIf{$j=i+k$}{
				$Y[j][i]=1$\;
			}{
				$Y[j][i]=0$\;
			}			
		}
		\tcp{$index[i]$ is the index of the first non-zero element from the left in the $i$-th row of $Y$}
		$index[i] = i$;
	}
	\For{$i=1$;$i\leq n$;$i++$}{
		\tcp{\textbf{Step2}: generate $W_{i}$ and $Y_{i}$}
		\For{$j=1$;$j\leq k$;$j++$}{
			$W[i \mod 2k][j]=-(x_{k-1}W[(i-1+2k) \mod 2k][j]+x_{k-2}W[(i-2+2k) \mod 2k][j]+\cdots+x_{-k}W[(i+1) \mod 2k][j])/x_{k}$\;
			$Y[i \mod k][j]= W[i \mod 2k][j]$\;
		}
		update $index[i \mod k]$\;
		$cur=i \mod k$\;
		\tcp{\textbf{Step3}: Gaussian elimination}
		\While{$\exists l, index[l]=index[cur]$}{
			\If{$l < cur \leq (i \mod n)$ or $(i \mod n) \ge l < cur$ or $l > (i \mod n)$, $cur \le (i \mod n)$}{
				exchange $l$ and $cur$\;
			}
			use $Y[l]$ to perform Gaussian elimination on $Y[cur]$\;
			update $index[cur]$\;
		}
		\eIf{$\exists l$, $index[l]=0$}{
			$R[i]=0$\;			
		}{
			$R[i]=1$\;
		}
	}
	return R\;		
	\caption{calculation for the invertibility sequence of banded Toeplitz matrices}
\end{algorithm}

We continuously remove the first row of $Y_{i}$ and add the last row of $W_{i+1}$ to construct $Y_{i+1}$. Since the first $k-1$ rows of $Y_{i+1}$ are already in quasi-row-echelon form, we need at most $k$ Gaussian eliminations to transform  $Y_{i+1}$ into quasi-row-echelon form, rather than $k^2$ eliminations. Since we always use the larger-index rows to perform Gaussian elimination on the small-index ones during the Gaussian elimination, the $j$-th row of $Y_{i+1}$ is only linearly expressed by the $j$-th to $k$-th rows of $W_{i+1}$. Therefore, when we use $W_{i+1}$ and  $Y_{i}$ to construct $Y_{i+1}$, we can remove the first row of  $Y_{i}$ and ensure that $Y_{i+1}$ and $W_{i+1}$ have the same invertibility, as shown in Eq. \ref{WY}, where $V_l$ is the $l$-th row of $W_i$ and function $L$ represents linear combination.

\begin{equation}
	\label{WY}
	W_i=\begin{bmatrix}
		V_1 \\
		V_2 \\
		\vdots \\
		V_k
	\end{bmatrix},
	Y_i=\begin{bmatrix}
		L(V_1,V_2,V_3,\cdots,V_k) \\
		L(V_2,V_3,\cdots,V_k) \\
		\vdots \\
		V_k
	\end{bmatrix}.
\end{equation}

Algorithm \ref{a1} shows the calculations in detail.

The time cost analysis of this algorithm is as follows:
\begin{itemize}
	\item \textbf{Step 1} (lines 1 to 17) is for initialization and involves no computation.
	\item \textbf{Step 2} (lines 19 to 24) calculates each element with a time cost of $2k+1$. Since there are $kn$ elements, the total time cost for this step is $2k^2n+kn$.
	\item \textbf{Step 3} (lines 25 to 36) requires at most $k$ Gaussian eliminations to transform each $Y$ into quasi-row-echelon form, with a time cost of $k$ for each cycle. As this while-loop is repeated $n$ times, the total cost for this step is $k^2n/2$.
\end{itemize}
In summary, the total time cost of the program is $5k^2n/2+kn$. This algorithm is more space-efficient: the size of $W$ is $2k*k$, and the size of $Y$ is $k*k$, therefore, the total space consumption of the algorithm is only $3k^2$.

\end{proof}

Relying on Theorem \ref{t1} and Algorithm \ref{a1}, we will be able to efficiently solve the invertibility of a class of banded Toeplitz matrices, providing a foundation for efficiently batch processing this class of banded Toeplitz matrices.

\section{Conclusion\label{s6}}

We equate the invertibility of an $n$-order banded Toeplitz matrix with bandwidth $2k+1$ to that of a small $k*k$ matrix. This allows us to calculate the invertibility sequence of a class of banded Toeplitz matrices with a time complexity of $5k^2n/2+kn$. This algorithm can be applied in solving equation systems and inverses of banded Toeplitz matrices.

\section*{Acknowledgments}
This study is financed by Tianjin Science and Technology Bureau, finance code: 21JCYBJC00210.

\section*{Declaration of generative AI}
Generative AI is only used for translation and language polishing in this paper.



%% References

%\bibliographystyle{unsrt}
\bibliographystyle{elsarticle-harv}
\bibliography{ref}






% \end{thebibliography}

\end{document}